\documentclass[12pt]{amsart}
\newtheorem{lemma}{Lemma}[section]
\newtheorem{proposition}[lemma]{Proposition}
\newtheorem{theorem}[lemma]{Theorem}

\newtheorem{remark}[lemma]{Remark}

\usepackage{latexsym}
\usepackage{amsmath}
\usepackage{amsfonts}
\usepackage{amssymb}
\renewcommand{\epsilon}{\varepsilon }
\newcommand{\R}{\mathbb{R}}
\newcommand{\Sn}{\mathbb{S}}

\begin{document}
\title[Infinitely many solutions]%
{Infinitely many solutions for the prescribed scalar curvature
problem on $\Sn^N$}

\author{  Juncheng Wei}

\address{Department of Mathematics,
  The Chinese University of Hong Kong,
  Shatin, Hong Kong }
\email{wei@math.cuhk.edu.hk}

\author{Shusen Yan}
\address{
School  of Mathematics, Statistics and Computer Science, The
University of New
England,  Armidale, NSW 2351,   Australia}
\email{syan@turing.une.edu.au}

\date{}

%\maketitle

\begin{abstract}

We  consider the following prescribed scalar curvature problem on $ \Sn^N$
\[ (*) \ \ \left\{\begin{array}{l}
- \Delta_{\Sn^N} u + \frac{N(N-2)}{2} u = \tilde{K} u^{\frac{N+2}{N-2}} \ \mbox{on} \ \Sn^N,
\\
u >0
\end{array}
\right.
\]
where $ \tilde{K}$ is positive and rotationally symmetric.
 We show that if $\tilde{K}$ has a local maximum point  between the poles then equation (*) has {\bf infinitely many non-radial positive} solutions, whose energy can be made arbitrarily large.

\end{abstract}

\maketitle
\baselineskip 18pt

%%%%%%%%%%%%%%%%%%%%%%%%SSSSSSSSSSSSSSSSSSSSSSSSSSSSSS
%%%%%%%%%%%%%%%%%%%%%%%%SSSSSSSSSSSSSSSSSSSSSSSSSSSSSS
\section{Introduction}
\setcounter{equation}{0}

Consider the standard $N$-sphere $(\Sn^N, g_0)$, $N \geq 3$. Let $ \tilde{K}$ be a fixed smooth function. The prescribed curvature problem asks if one can find a conformally invariant metric $g$ such that the scalar curvature becomes $\tilde{K}$. The problem consists in solving the following equation on $\Sn^N$:
\begin{equation}
\label{1.4n}
\left\{\begin{array}{l}
- \Delta_{\Sn^N} u + \frac{N(N-2)}{2} u + \tilde{K} u^{\frac{N+2}{N-2}}=0 \ \mbox{on} \ \Sn^N,
\\
u >0.
\end{array}
\right.
\end{equation}

 Problem  (\ref{1.4n}) does not always admit a solution. A first necessary condition for the existence  is that $ \max_{\Sn^N} \tilde{K}>0$, but there are also some obstructions, which are said of {\em topological type}. For example, a necessary condition is the following  Kazdan-Warner condition:
\begin{equation}
\label{ka}
\int_{\Sn^N} \nabla \tilde{K} \cdot \nabla x u^{\frac{2N}{N-2}}=0.
\end{equation}

 The problem of determining which $\tilde{K}$  admits a solution to
(\ref{1.4n})  has been studied extensively. See \cite{AAP},
\cite{bc}--\cite{cl2}, \cite{DN}--\cite{LL}, \cite{N, NY, Y}   and
the references therein.  Some existence results have been obtained
under some assumptions involving the Laplacian at the critical point
of $\tilde{K}$, see Chang-Yang \cite{CY}, Bahri-Coron \cite{bc} and Schoen-Zhang \cite{sz} for the case
$N=3$, and Y. Li \cite{L4} for the case $ N \geq 4$. For example, in Bahri
and Coron \cite{bc},
 it is assumed that $ \tilde{K}$ is a positive
Morse function with $\Delta \tilde{K} (x) \not =0$ if $\nabla
\tilde{K} (x) =0$, then if $m(x)$ denotes the Morse index of the
critical point $x$ of $K$, (\ref{1.4n}) has a solution provided that
\[ \sum_{ \nabla \tilde{K} =0, \Delta \tilde{K} (x) <0} (-1)^{m(x)} \not = -1.\]
The result has been extended to any $\Sn^N, N \geq 3$ by Y.Li in  \cite{L3}-\cite{L4}. Roughly, it is assumed that there exists $\beta,  N-2<\beta <N$ such that
\begin{equation}
\tilde{K} (\xi)= \tilde{K}(\xi_0) + \sum_{j=1}^N a_j |\xi_j-\xi_{0, j}|^\beta + h.o.t.
\end{equation}
where $a_j \not =0, \sum_{j=1}^N a_j \not =0$. Let $ \Sigma= \{ \xi: \nabla \tilde{K} (\xi)=0, \sum_{j=1}^N a_j <0\}$ and $ i(\xi)$ be the number of $ a_j$ such that $\tilde{K} (\xi)=0, a_j <0$. Then (\ref{1.4n}) has a solution provided
\begin{equation}
\sum_{\xi \in \Sigma} (-1)^{i (\xi)} \not = (-1)^N.
\end{equation}

  By using the stereo-graphic  projection,
the prescribed scalar curvature problem (\ref{1.4n}) can be reduced to \eqref{1.4}
\begin{equation}
\label{1.4}
\begin{cases}
 -\Delta u =K(y)u^{\frac{N+2}{N-2}}, u >0, & y\in  \R^N \\
  u \in D^{1,2}(\R^N)
\end{cases}
\end{equation}
where $D^{1,2} (\R^N)$ denotes the completion of $C_0^\infty (\R^N)$ under the norm $ \int_{\R^N} |\nabla u|^2$.

Much less  is known  about the multiplicity of the solutions of (\ref{1.4}).  Amrosetti, Azorero and Peral \cite{AAP}, and  Cao, Noussair and Yan \cite{CNY}  proved the existence of two or many solutions if  $K$ is a perturbation of the constant, i.e.
\begin{equation}
\tilde{K}= K_0 + \epsilon h(x), 0<\epsilon <<1.
\end{equation}  
On the other hand,
 Y. Li proved in \cite{L1} that
\eqref{1.4} has infinitely many solutions if $K(x)$ is periodic,
while similar result was obtained in \cite{Y} if $K(x)$ has a
sequence of strict local maximum points tending to infinity. Note
that this condition for $K(x)$ at the infinity implies that the
corresponding function $\tilde{K}$ defined on $\Sn^N$ has a singularity at the
south pole.

In this paper,  we consider the simplest  case, i.e., $ \tilde{K}$ is rotationally symmetric, $ K=K(r), r=|y|$.  It follows from the Pohozaev identity (\ref{ka})  that \eqref{1.4} has no
solution if $K'(r)$ has fixed sign. Thus we assume that {\em $K$ is positive and not monotone}.  On the other hand, Bianchi
\cite{B1} showed that any solution of \eqref{1.4} is radially symmetric
 if there is a $r_0>0$, such that $K(r) $ is non-increasing in $(0,r_0]$, and
non-decreasing in $[r_0,+\infty)$. Moreover, in \cite{BE}, it was
proved that \eqref{1.4} has no  solutions for some function
$K(r)$, which  is non-increasing in $(0,1]$, and non-decreasing in
$[1,+\infty)$. Therefore, we see that to obtain a solution for \eqref{1.4},
it is natural to assume that {\em $K(r)$ has a local maximum at
$r_0>0$}. The purpose of this paper is to answer the following two questions:

\medskip

\noindent
{\bf \it Q1: Does the existence of a local maximum of $K$ guarantee the existence of a solution to (\ref{1.4})?}

\medskip

\noindent
{\bf \it Q2: Are there non-radially symmetric solutions  to (\ref{1.4})?}

(Question $Q2$ has been asked by Bianchi \cite{B1}.)

\medskip

The aim of this paper is to show that if  $K(r)$ has a local maximum
at $r_0>0$, then \eqref{1.4} has {\bf infinitely many non-radial
solutions}.  This answers Q1 and Q2 affirmatively.   As far as we know, we believe our result is the first on the existence of infinitely many solution for (\ref{1.4}).

 We assume that $K(r)$  satisfies the following condition:

(K):  There is a  constant $r_0>0$, such that
\[
K(r)= K(r_0)- c_0|r-r_0|^m+ O(|r-r_0|^{m+\theta}),\quad r\in (r_0-\delta, r_0+\delta),
\]
where $c_0>0$, $\theta>0$ are some constants, and the constant $m$
satisfies  $m\in [2,N-2)$.

Without loss of generality, we assume that
\[
K(r_0)=1.
\]

Our main result in this paper can be stated as follows:

\vskip 0.5cm

\begin{theorem}
\label{main} Suppose that $N\ge 5$.  If $K(r)$ satisfies (K), then
problem (\ref{1.4}) has infinitely many non-radial solutions.
\end{theorem}

\medskip

\begin{remark}
  Theorem~\ref{main} shows that the condition in \cite{B1} is
optimal. We shall prove Theorem~\ref{main} by constructing solutions
with large number of bubbles lying near the sphere $|y|=r_0$. So the
energy of these solutions can be made arbitrary large and the
distance  between different bubbles can be made arbitrary small.
When $ N=3, 4$, we know that the energy of the solutions to
(\ref{1.4}) is uniformly bounded and the distance between bubbles is
uniformly bounded  from below. See \cite{sz} (for $N=3$) and Theorem 0.10
of \cite{L4} (for $N=4$). On the other hand, if $ K (y)= K(y_0)
+O(|y-y_0|^m)$ where $ m \in [N-2, N)$ for $N \geq 5$,   the energy
of solutions is also be bounded. See \cite{L4}. So our assumptions
on $N$ and $ m$  are almost optimal in the construction of the
solutions in this paper.

\end{remark}

\begin{remark}
 The radial symmetry can be replaced by the following  weaker
symmetry assumption: after suitably rotating the coordinate system,

\begin{itemize}
\item[(K1)]
$  K(y)= K(y^{'}, y^{''}) = K(|y^{'}|, |y_{N_0+1}|, ...,
|y_{N}|), \ \mbox{where} \   y=(y^{'}, y^{''}) \in \R^{2} \times
\R^{N-2}$,

\item[(K2)]
$
  K(y)= K(y_0)- c_0 |y -y_0|^m + O( |y-y_0 |^{m+\theta}),
|y^{'}| \in (|y_0^{'}| -\delta, |y_0^{'}| +\delta), |y''|\le
\delta,
$
 where $y_0=(y'_0,0)$.
\end{itemize}

\end{remark}

\medskip

\begin{remark}

Theorem \ref{main} exhibits a {\bf new phenomena} for the prescribed scalar curvature problem. It suggests that if the critical points of $K$ are not isolated, new solutions to (\ref{1.4}) may bifurcate. We formulate the following conjecture in the general case.

\end{remark}

\medskip

\noindent {\it {\bf Conjecture:} Assume that the set $ \{ K(x)=
\max_{x \in \R^N} K(x) \}$ is an
 $l$-dimensional smooth manifold without
boundary, where $ 1\leq l \leq N-1$. The problem (\ref{1.4}) admits infinitely many positive
solutions.}

Before we close this introduction, let us outline the main idea in the proof of Theorem~\ref{main}.

Let us fix a positive integer
\[
 k \geq k_0,
 \]
where $k_0$ is large, to be determined  later.

Set
\[
\mu= k^{\frac{N-2}{N-2-m}},
\]
to be the scaling parameter.

Let $ 2^{*}=\frac{2N}{N-2}$. Using the transformation  $ u(y)\mapsto \mu^{-\frac{N-2}2}
u\bigl(\frac y\mu\bigr)$,  we find that \eqref{1.4} becomes

\begin{equation}
\label{n1.4}
\begin{cases}
 -\Delta u =K\bigl(\frac {|y|}\mu\bigr)u^{2^*-1}, u >0, & y\in  \R^N, \\
  u \in D^{1,2}(\R^N).
\end{cases}
\end{equation}

 It is well-known   that   the functions
$$
U_{x,\Lambda}(y)=\bigl(N(N-2)\bigr)^{\frac{N-2}4}\left(\frac{\Lambda }{1+\Lambda^2|y-x|^2}\right)^{\frac{N-2}{2}},
~~~ \mu>0,~~ x\in \R^N
$$
are the only solutions to the problem
$$
-\Delta u = u^{\frac{N+2}{N-2}}, ~~~~ u>0 ~~ \mbox{in  } \R^N.
$$

Let $y=(y',y'')$, $y'\in \R^2$, $y'' \in \R^{N-2}$.  Define

\[
\begin{split}
H_s=\bigl\{ u: & u\in D^{1,2}(\R^N), u\;\text{is even in} \;y_h, h=2,\cdots,N,\\
& u(r\cos\theta , r\sin\theta, y'')=
u(r\cos(\theta+\frac{2\pi j}k) , r\sin(\theta+\frac{2\pi j}k), y'')
\bigr\}.
\end{split}
\]

Let

\[
x_j=\bigl(r \cos\frac{2(j-1)\pi}k, r\sin\frac{2(j-1)\pi}k,0\bigr),\quad j=1,\cdots,k,
\]
where $0$ is the zero vector in $\R^{N-2}$, and let
\[
W_r(y)=\bigl(N(N-2)\bigr)^{\frac{N-2}4}\sum_{j=1}^k \frac{\Lambda^{\frac{N-2}2}}{(1+\Lambda^2|y-x_j|^2)^{\frac{N-2}2}}.
\]

In this paper, we always assume that

\[
r\in [r_0\mu-\frac1{\mu^{\bar\theta}},  r_0\mu
+\frac1{\mu^{\bar\theta}}], \quad\text{ for some small }\;\bar\theta>0,
\]
and

\[
L_0\le\Lambda\le L_1,\quad \text{ for some constants }\;L_1>L_0>0.
\]

Theorem~\ref{main} is a direct consequence of the following result:

\begin{theorem}
\label{th11} Suppose that $N\ge 5$.  If $K(r)$ satisfies (K), then
there  is an integer $k_0>0$, such that for any integer  $k\ge k_0$,
\eqref{n1.4} has a solution $u_k$ of the form
\[
u_k = W_{r_k}(y)+\omega_k,
\]
where  $\omega_k\in H_s$, and as $k\to +\infty$,
$\|\omega\|_{L^\infty}\to 0$,  $r_k\in
[r_0\mu-\frac1{\mu^{\bar\theta}},  r_0\mu
+\frac1{\mu^{\bar\theta}}]$ .

\end{theorem}

We will use the techniques in the singularly perturbed elliptic
problems to prove Theorem~\ref{th11}.  We know that there is always
a small parameter in a singularly perturbed elliptic problem. Although
there is no parameter in \eqref{1.4}, we use $k$, {\bf the number of the
bubbles} of the solutions, as the parameter in the construction of
bubbles solutions for \eqref{1.4}. This is the  {\bf new idea} of this
paper. This is partly motivated by recent paper of Lin-Ni-Wei \cite{LNW} where they constructed multiple spikes to a singularly perturbed problem.  There they allowed the number of spikes to depend on the small parameter.

The main difficulty in constructing solution with $k$-bubbles is
that we need to obtain a better control of the error terms. Since
the number of the bubbles is large, it is very hard to carry out the
reduction procedure by using the standard norm as in \cite{Ba,R}.
Noting that the maximum norm will not be affected by the number of
the bubbles, we will carry out the reduction procedure in a space
with weighted maximum norm. Similar  weighted maximum norm has been
used in \cite{DFM},\cite{RW1}--\cite{RW3}. But the estimates in the
reduction procedure in this paper are much more complicated than
those in \cite{DFM},\cite{RW1}--\cite{RW3}, because the number of
the bubbles  is large.

\medskip

\noindent
{\bf Acknowledgment.} The first
author is supported by an Earmarked
Grant  from RGC  of Hong Kong. The second author is partially supported
by ARC.

\section{Finite-dimensional Reduction}
\setcounter{equation}{0}

 In this section, we perform a finite-dimensional reduction.

 Let
\begin{equation}
\|u\|_{*} = \sup_{y \in \R^N }\Bigl(
\sum_{j=1}^k \frac1{(1+ |y-x_j|)^{ \frac{N-2}2+\tau
}}\Bigr)^{-1}|u(y)| ,
\end{equation}
and

\begin{equation}
\|f\|_{**} = \sup_{y \in \R^N }\Bigl(
\sum_{j=1}^k \frac1{(1+ |y-x_j|)^{ \frac{N+2}2+\tau
}}\Bigr)^{-1}|f(y)|,
\end{equation}
where $\tau=1+\bar \eta$ and $\bar\eta>0$ is small.

Let
 \[
Z_{i,1}=\frac{\partial U_{x_i,\Lambda}}{\partial r},\quad
 Z_{i,2}=\frac{\partial U_{x_i,\Lambda}}{\partial \Lambda} .
  \]

 Consider

\begin{equation}\label{lin}
\begin{cases}
-\Delta \phi_k -(2^*-1)K\bigl(\frac{|y|}{\mu}\bigr)
W_r^{2^*-2}\phi_k=&h_k+c_{1}\sum_{i=1}^k U_{x_i,\Lambda}^{2^*-2}Z_{i,1} +c_2 \sum_{i=1}^k U_{x_i,\Lambda}^{2^*-2}Z_{i,2},
\;\;
\text{in}\;
\R^N,\\
\phi_k\in H_s,\\
< U_{x_i,\Lambda}^{2^*-2}Z_{i, l},\phi_k>=0  & i=1,\cdots,k,\; l=1,2
\end{cases}
\end{equation}
for some numbers $c_{i}$, where $<u,v>=\int_{\R^N}uv$.

\begin{lemma}\label{l1}
Assume that $\phi_k$ solves (\ref{lin}) for $h=h_{k}$. If $\Vert
h_{k}\Vert_{**}$ goes to zero as $k$ goes to infinity, so does
$\Vert \phi_k\Vert_{*}$.
\end{lemma}

\begin{proof}

We argue by contradiction.  Suppose that there are $k\to +\infty$,
$h=h_{k}$, $\Lambda_k\in [L_1,L_2]$, $r_k\in
[r_0\mu-\frac1{\mu^{\bar\theta}}, r_0\mu
+\frac1{\mu^{\bar\theta}}]$, and $\phi_k$ solving (\ref{lin}) for
$h=h_{k}$, $\Lambda=\Lambda_k$, $r=r_k$,  with
 $\Vert
h_{k}\Vert_{**}\to 0$, and $\|\phi_k\|_*\ge c'>0$.  We may assume
that $\|\phi_k\|_*=1$. For simplicity, we drop the subscript $k$.

We rewrite \eqref{lin} as

\begin{equation}\label{g}
\begin{split}
\phi(y)=&(2^*-1)\int_{\R^N}\frac1{|z-y|^{N-2}}K\bigl(\frac{|z|}{\mu}\bigr)
W_{r}^{2^*-2}\phi(z)\,dz\\
&+\int_{\R^N}\frac1{|z-y|^{N-2}}\bigl(
h(z)+c_{1}\sum_{i=1}^kZ_{i,1}(z)U_{x_i,\Lambda}^{2^*-2}(z) +c_2 \sum_{i=1}^k Z_{i,2}(z)U_{x_i,\Lambda}^{2^*-2}(z)
\bigr)\,dz.
\end{split}
\end{equation}

Using Lemma~\ref{laa3}, we have

\begin{equation}\label{g1}
\begin{split}
&\Bigl|(2^*-1)\int_{\R^N}\frac1{|z-y|^{N-2}}K\bigl(\frac{|z|}{\mu}\bigr)
W_{r}^{2^*-2}\phi(z)\,dz\Bigr|\\
\le & C \|\phi\|_* \int_{\R^N}\frac1{|z-y|^{N-2}}
W_{r}^{2^*-2}\sum_{j=1}^k \frac1{(1+|z-x_j|)^{\frac{N-2}2+\tau}}\,dz\\
\le & C \|\phi\|_*\bigl(\sum_{j=1}^k \frac1{(1+|y-x_j|)^{\frac{N-2}2+\tau+\theta}}+o(1)
\sum_{j=1}^k \frac1{(1+|y-x_j|)^{\frac{N-2}2+\tau}}\bigr).
\end{split}
\end{equation}

It follows from Lemma~\ref{laa2} that

\begin{equation}\label{g2}
\begin{split}
&\Bigl|\int_{\R^N}\frac1{|z-y|^{N-2}}
h(z)
\,dz\Bigr|\\
\le & C\|h\|_{**}\int_{\R^N}\frac1{|z-y|^{N-2}}\sum_{j=1}^k \frac{1}
{(1+|z-x_j|)^{\frac{N+2}2+\tau}}\,dz\\
\le & C\|h\|_{**}\sum_{j=1}^k \frac{1}
{(1+|y-x_j|)^{\frac{N-2}2+\tau}},
\end{split}
\end{equation}
and

\begin{equation}\label{g3}
\begin{split}
&\Bigl|\int_{\R^N}\frac1{|z-y|^{N-2}}\sum_{i=1}^k
Z_{i,l}(z)U_{x_i,\Lambda}^{2^*-2}(z)
\,dz\Bigr|\\
\le & C\sum_{i=1}^k\int_{\R^N}\frac1{|z-y|^{N-2}} \frac{1}
{(1+|z-x_i|)^{N+2}}\,dz
\le  C \sum_{i=1}^k \frac{1}
{(1+|y-x_i|)^{\frac{N-2}2+\tau}}.
\end{split}
\end{equation}

Next, we estimate  $c_l$, $l=1,2$. Multiplying \eqref{lin} by
$Z_{1,l}$ and integrating, we see that $c_t$ satisfies

\begin{equation}\label{g6}
\sum_{t=1}^2\sum_{i=1}^k \bigl\langle  U_{x_i,\Lambda}^{2^*-2} Z_{i,t}, Z_{1,l} \bigr\rangle c_t=
\bigl\langle
-\Delta \phi
-(2^*-1)K\bigl(\frac{|y|}{\mu}\bigr)
W_{r}^{2^*-2}\phi, Z_{1,l} \bigr\rangle- \bigl\langle h, Z_{1,l} \bigr\rangle.
\end{equation}

It follows from Lemma~\ref{laa1} that

\[
\begin{split}
&\bigl|\bigl\langle h, Z_{1,l} \bigr\rangle\bigr| \le C\|h\|_{**}
\int_{\R^N}\frac1{(1+|z-x_1|)^{N-2}}\sum_{j=1}^k \frac1{(1+|z-x_j|)^{\frac{N+2}2+\tau}}\,dz\\
\le & C \|h\|_{**}.
\end{split}
\]

 On the other hand,

\begin{equation}\label{g4}
\begin{split}
& \bigl\langle
-\Delta \phi
-(2^*-1)K\bigl(\frac{|z|}{\mu}\bigr)
W_{r}^{2^*-2}\phi, Z_{1,l} \bigr\rangle\\
=&  \bigl\langle
-\Delta Z_{1,l}
-(2^*-1)K\bigl(\frac{|z|}{\mu}\bigr)
W_{r}^{2^*-2} Z_{1,l}, \phi\bigr\rangle\\
=& (2^*-1)\bigl\langle \bigl(1-
K\bigl(\frac{|z|}{\mu}\bigr)
W_{r}^{2^*-2} Z_{1,l} , \phi\bigr\rangle\\
= &\|\phi\|_* O\Bigl(
\int_{\R^N} \bigl|K\bigl(\frac{|z|}{\mu}\bigr)-1\bigr|
W_{r}^{2^*-2}(z) \frac1{(1+|z-x_1|)^{N-2}}\sum_{j=1}^k \frac1{(1+|z-x_j|)^{\frac{N-2}2+\tau}}\,dz\Bigr).
\end{split}
\end{equation}

Similar to the proof of Lemma~\ref{laa3}, we obtain

\[
\begin{split}
&
\int_{||z|-\mu r_0|\le \sqrt\mu} \bigl|K\bigl(\frac{|z|}{\mu}\bigr)-1\bigr|
W_{r}^{2^*-2}(z) \frac1{(1+|z-x_1|)^{N-2}}\sum_{j=1}^k \frac1{(1+|z-x_j|)^{\frac{N-2}2+\tau}}\,dz\\
\le & \frac C{\sqrt\mu}\int_{\R^N}
W_{r}^{2^*-2}(z) \frac1{(1+|z-x_1|)^{N-2}}\sum_{j=1}^k \frac1{(1+|z-x_j|)^{\frac{N-2}2+\tau}}\,dz\\
\le & \frac C{\sqrt\mu},
\end{split}
\]
and

\[
\begin{split}
&
\int_{||z|-\mu r_0|\ge \sqrt\mu} \bigl|K\bigl(\frac{|y|}{\mu}\bigr)-1\bigr|
W_{r}^{2^*-2}(z) \frac1{(1+|z-x_1|)^{N-2}}\sum_{j=1}^k \frac1{(1+|z-x_j|)^{\frac{N-2}2+\tau}}\,dz\\
\le & \frac C{\mu^\sigma}\int_{\R^N}
W_{r}^{2^*-2}(z) \frac1{(1+|z-x_1|)^{N-2}}\sum_{j=1}^k \frac1{(1+|z-x_j|)^{\frac{N-2}2+\tau-2\sigma}}\,dz\\
\le & \frac C{\mu^\sigma},
\end{split}
\]
since    if $ ||z|-\mu r_0|\ge \sqrt\mu$, then

\[
||z|-|x_1||\ge  ||z|-\mu r_0| -||x_1| -\mu r_0| \ge
\sqrt\mu-\frac1{\mu^{\bar\theta}}\ge \frac12\sqrt\mu.
\]
Thus,

\[
\begin{split}
&
\int_{\R^N} \bigl|K\bigl(\frac{|z|}{\mu}\bigr)-1\bigr|
W_{r}^{2^*-2}(z) \frac1{(1+|z-x_1|)^{N-2}}\sum_{j=1}^k \frac1{(1+|z-x_j|)^{\frac{N-2}2+\tau}}\,dz\\
\le & \frac C{\mu^\sigma},
\end{split}
\]
which, together with \eqref{g4}, gives

\begin{equation}\label{g5}
\bigl\langle
-\Delta \phi
-(2^*-1) K\bigl(\frac{|z|}{\mu}\bigr)
W_{r_k}^{2^*-2}\phi, Z_{1,l} \bigr\rangle
= \|\phi\|_* O\Bigl(\frac 1{\mu^\sigma}\Bigr).
\end{equation}

But there is a constant $\bar c>0$,

\[
\sum_{i=1}^k \bigl\langle  U_{x_i,\Lambda}^{2^*-2} Z_{i,t}, Z_{1,l} \bigr\rangle=
(\bar c+o(1)) \delta_{tl}.
\]
Thus we obtain from \eqref{g6} that

\[
c_l=O\Bigl(\frac 1{\mu^\sigma}\|\phi\|_{*}+\|h\|_{**} \Bigr).
\]
So,

\begin{equation}\label{g8}
\|\phi\|_{*}\le \Bigl( o(1) + \|h_k\|_{**} +\frac{
\sum_{j=1}^k \frac1{(1+|y-x_j|)^{\frac{N-2}2+\tau+\theta}}}{
\sum_{j=1}^k \frac1{(1+|y-x_j|)^{\frac{N-2}2+\tau}}}\Bigr).
\end{equation}

Since $\|\phi\|_*=1$, we obtain from \eqref{g8} that there is $R>0$,
such that

\begin{equation}\label{g10}
\|\phi(y)\|_{B_R(x_i)}\ge a>0,
\end{equation}
for some $i$. But $\bar \phi(y)=\phi(y-x_i)$ converges uniformly in
any compact set to a solution $u$ of

\begin{equation}\label{g9}
-\Delta u-(2^*-1) U_{0,\Lambda}^{2^*-2} u=0,\quad \text{in}\;\R^N,
\end{equation}
for some $\Lambda\in [L_1,L_2]$, and $u$ is perpendicular to the
kernel of \eqref{g9}. So, $u=0$. This is a contradiction to
\eqref{g10}.

\end{proof}

From Lemma~\ref{l1}, using the same argument as in the proof of
Proposition~4.1 in \cite{DFM}, we can prove the following result :
\begin{proposition}\label{p1}
There exists $k_0>0 $ and a constant $C>0$, independent of $k$, such
that for all $k\ge k_0$ and all $h\in L^{\infty}(\R^N)$, problem
$(\ref{lin})$ has a unique solution $\phi\equiv L_k(h)$. Besides,
\begin{equation}\label{Le}
\Vert L_k(h)\Vert_*\leq C\Vert h\Vert_{**},\qquad
|c_{l}|\leq C\Vert h\Vert_{**}.
\end{equation}

\end{proposition}

 Now, we consider

\begin{equation}\label{re}
\begin{cases}
-\Delta \bigl(  W_{r} +\phi \bigr)
=K\bigl(\frac{y}{\mu}\bigr)\bigl( W_{r} +\phi \bigr)^{2^*-1}
+& \sum_{t=1}^2 c_{t}\sum_{i=1}^k U_{x_i,
\Lambda}^{2^*-2}Z_{i,t}, \; \text{in}\;
\R^N,\\
\phi_k\in H_s,\\
<U_{x_i,\Lambda}^{2^*-2}Z_{i,l},\phi_k>=0,  & i=1,\cdots,k,\; l=1,2.
\end{cases}
\end{equation}
We have

\begin{proposition}\label{p1-6-3}
There is an integer $k_0>0$, such that for each $k\ge k_0$, $L_0\le
\Lambda\le L_1$, $|r-\mu r_0|\le \frac1{\mu^{\bar\theta}}$, where
$\bar\theta>0$ is a fixed small constant, \eqref{re} has  a unique
solution $\phi=\phi(r,\Lambda)$, satisfying

\[
\|\phi\|_{*}\le C\bigl(\frac k\mu\bigr)^{\frac {N+2}2-\tau},\qquad |c_t|\le
C\bigl(\frac k\mu\bigr)^{\frac {N+2}2-\tau}.
\]

\end{proposition}

Rewrite \eqref{re} as

\begin{equation}\label{re1}
\begin{cases}
-\Delta \phi
-(2^*-1)K\bigl(\frac{|y|}{\mu}\bigr)
 W_{r}^{2^*-2}\phi = & N(\phi)+l_k
+ \sum_{t=1}^2 c_{i}\sum_{i=1}^k  U_{x_i,\Lambda}^{2^*-2}Z_{i,t}, \;
\text{in}\;
\R^N,\\
\phi\in H_s,\\
<U_{x_i,\Lambda}^{2^*-2}Z_{i,l },\phi>=0,  & i=1,\cdots,k,\; l=1,2,
\end{cases}
\end{equation}
where

\[
N(\phi)=K\bigl(\frac{|y|}{\mu}\bigr)\Bigl(\bigl(  W_{r} +\phi \bigr)^{2^*-1}-
 W_{r}^{2^*-1}-(2^*-1) W_{r}^{2^*-2}\phi\Bigr),
\]
and

\[
l_k=K\bigl(\frac{|y|}{\mu}\bigr)  W_{r}^{2^*-1}-\sum_{j=1}^k U_{x_j,\Lambda}^{2^*-1}.
\]

In order to use the contraction mapping theorem to prove that
\eqref{re1} is uniquely solvable  in the set that $\|\phi\|_*$ is
small, we need to estimate $N(\phi)$ and $l_k$.

\begin{lemma}\label{l-1-5-3}

If $N\ge 6$, then

\[
\|N(\phi)\|_{**}\le C k^{\frac{4}{N-2}} \|\phi\|_*^{2^*-1}.
\]
If $N=5$,
\[
\|N(\phi)\|_{**}\le C\|\phi\|_*^2.
\]

\end{lemma}

\begin{proof}

We have

\[
|N(\phi)|\le
\begin{cases}
C|\phi|^{2^*-1}, & N\ge 6;\\
CW_{r}^{\frac13}\phi^2, &N=5.
\end{cases}
\]

Firstly, we consider $N\ge 6$. For any $p>1$, the function $t^p$ is
convex in $t>0$. Thus

\begin{equation}\label{1-8-4}
\bigl(\sum_{j=1}^k \frac{|a_j|}{k}\bigr)^p \leq \sum_{j=1}^k  \frac{|a_j|^p}{k}.
\end{equation}

Using \eqref{1-8-4}, we obtain

\begin{equation}\label{1-l-1-5-3}
\begin{split}
& |N(\phi)|\le  C\|\phi\|_*^{2^*-1}\Bigl(
\sum_{j=1}^k \frac1{(1+|y-x_j|)^{\frac{N-2}2+\tau}}\Bigr)^{2^*-1}\\
\le & C\|\phi\|_*^{2^*-1} k^{\frac{4}{N-2}}  \sum_{j=1}^k \frac1{(1+|y-x_j|)^{\frac{N+2}2+ \frac{N+2}{N-2} \tau}} \\
\le & C\|\phi\|_*^{2^*-1} k^{\frac{4}{N-2}} \sum_{j=1}^k \frac1{(1+|y-x_j|)^{\frac{N+2}2+  \tau}}.
\end{split}
\end{equation}

Thus, the result follows.

It remains to prove the result for $N=5$. We have

\[
|N(\phi)|\le  C\|\phi\|_*^{2}\sum_{i=1}^k \frac1{1+|y-x_i|}
\Bigl(
\sum_{j=1}^k \frac1{1+|y-x_j|)^{\frac{3}2+\tau}}\Bigr)^{2}.
\]

Define

\[
\Omega_j=\bigl\{y: \; y=(y', y'')\in\R^2\times\R^{N-2},
\bigl\langle \frac {y'}{|y'|}, \frac{x_j}{|x_j|}\bigr\rangle\ge \cos\frac{\pi}{k}\bigr\}.
\]
Without loss of generality, we assume $y\in\Omega_1$. Then

\[
|y-x_j|\ge |y-x_1|,\quad j=2,\cdots,k.
\]
So,

\[
\begin{split}
&\sum_{i=2}^k \frac1{1+|y-x_i|}\le \frac1{(1+|y-x_1|)^{\frac12}}\sum_{i=2}^k \frac1{(1+|y-x_i|)^{\frac12}}\\
\le & \frac C{(1+|y-x_1|)^{\frac23}}\sum_{i=2}^k \frac1{|x_1-x_i|^{\frac13}}\\
\le & \frac C{(1+|y-x_1|)^{\frac23}}\frac k{\mu^{\frac13}}\le \frac C{(1+|y-x_1|)^{\frac23}},
\end{split}
\]
since

\[
\frac k{\mu^{\frac13}}\le \frac {Ck}{k^{\frac 3{3-m}\frac13}}\le C.
\]
Similarly

\[
\sum_{j=2}^k \frac1{(1+|y-x_j|)^{\frac{3}2+\tau}}\le \frac C{(1+|y-x_1|)^{\frac32+\tau-\frac13}}.
\]
So, we have proved

\[
\begin{split}
&\sum_{i=1}^k \frac1{1+|y-x_i|}
\Bigl(
\sum_{j=1}^k \frac1{1+|y-x_j|)^{\frac{3}2+\tau}}\Bigr)^{2}\\
\le &
\frac C{(1+|y-x_1|)^{3+2\tau}}\le \frac C{(1+|y-x_1|)^{\frac 72+\tau}},\quad y\in\Omega_1,
\end{split}
\]
since $\tau>1$. Thus,

\[
\|N(\phi)\|_{**}\le C\|\phi\|_*^2.
\]

\end{proof}

Next, we estimate $l_k$.

\begin{lemma}\label{l-2-5-3}
Assume that  $\bigl| |x_1|-\mu r_0\bigr|\le \frac 1{\mu^{\bar
\theta}}$, where $\bar \theta>0$ is a fixed small constant.

 If $N\ge
5$, then

\[
\|l_k\|_{**}\le C \bigl(\frac{k}{\mu}\bigr)^{\frac {N+2}2-\tau}.
\]

\end{lemma}

\begin{proof}

Define

\[
\Omega_j=\bigl\{y: \; y=(y', y'')\in\R^2\times\R^{N-2},
\bigl\langle \frac {y'}{|y'|}, \frac{x_j}{|x_j|}\bigr\rangle\ge \cos\frac{\pi}{k}\bigr\}.
\]

We have

\[
\begin{split}
l_k=& K\bigl(\frac {|y|}\mu\bigr)\Bigl(
W_{r}^{2^*-1}-\sum_{j=1}^k U_{x_j,\Lambda}^{2^*-1}\Bigr)\\
&+\sum_{j=1}^k U_{x_j,\Lambda}^{2^*-1}\Bigl(
K\bigl(\frac {|y|}\mu\bigr)-1\Bigr)\\
=:& J_1+J_2.
\end{split}
\]

From the symmetry, we can assume that $y\in\Omega_1$. Then,

\[
|y-x_j|\ge |y-x_1|,\quad \forall\; y\in\Omega_1.
\]
Thus,

\begin{equation}\label{1-l2-5-3}
|J_1|\le C\frac1{(1+|y-x_1|)^4}\sum_{j=2}^k \frac1{(1+|y-x_j|)^{N-2}}+C\Bigl(
\sum_{j=2}^k \frac1{(1+|y-x_j|)^{N-2}}\Bigr)^{2^*-1}.
\end{equation}

Using Lemma~\ref{laa1}, we obtain

\begin{equation}\label{2-l2-5-3}
\begin{split}
&\frac1{(1+|y-x_1|)^4} \frac1{(1+|y-x_j|)^{N-2}}\\
\le & C\Bigl(\frac1{(1+|y-x_1|)^{\frac{N+2}2+\tau} }+\frac1{(1+|y-x_j|)^{\frac{N+2}2+\tau}}\Bigr)\frac1{|x_j-x_1|^{
\frac{N+2}2-\tau}}\\
\le& C\frac1{(1+|y-x_1|)^{\frac{N+2}2+\tau} }\frac1{|x_j-x_1|^{
\frac{N+2}2-\tau}},\quad j>1.
\end{split}
\end{equation}
Since $\tau<2$, we see $\frac{N+2}2-\tau>\frac{N-2}2>1$. Thus

\begin{equation}\label{3-l2-5-3}
\begin{split}
&\frac1{(1+|y-x_1|)^4} \sum_{j=2}^k \frac1{(1+|y-x_j|)^{N-2}}\\
\le & C\frac1{(1+|y-x_1|)^{\frac{N+2}2+\tau} }\bigl(\frac k\mu\bigr)^{
\frac{N+2}2-\tau}.
\end{split}
\end{equation}

On the other hand, for $y\in\Omega_1$, using Lemma~\ref{laa1} again,

\[
\begin{split}
&\frac1{(1+|y-x_j|)^{N-2}}\le  \frac1{(1+|y-x_1|)^{\frac{N-2}2}}\frac1{(1+|y-x_j|)^{\frac{N-2}2}}\\
\le & \frac C{|x_j-x_1|^{\frac{N-2}2-\frac{N-2}{N+2}\tau}}\Bigl(
\frac1{(1+|y-x_1|)^{\frac{N-2}2+\frac{N-2}{N+2}\tau}}+\frac1{(1+|y-x_j|)^{\frac{N-2}2+\frac{N-2}{N+2}\tau}}\Bigr)\\
\le & \frac C{|x_j-x_1|^{\frac{N-2}2-\frac{N-2}{N+2}\tau}}
\frac 1{(1+|y-x_1|)^{\frac{N-2}2+\frac{N-2}{N+2}\tau}}.
\end{split}
\]

If $N\ge 5$, $\tau=1+\bar\eta$ and $\bar\eta>0$ is small, then
$\frac{N-2}2-\frac{N-2}{N+2}\tau> 1$. Thus

\[
\sum_{j=2}^k \frac1{(1+|y-x_j|)^{N-2}}\le C\bigl(
\frac k\mu\bigr)^{\frac{N-2}2-\frac{N-2}{N+2}\tau} \frac 1{(1+|y-x_1|)^{\frac{N-2}2+\frac{N-2}{N+2}\tau}},
\]
which, gives

\[
\Bigl(\sum_{j=2}^k \frac1{(1+|y-x_j|)^{N-2}}\Bigr)^{2^*-1}\le C\bigl(
\frac k\mu\bigr)^{\frac{N+2}2-\tau} \frac 1{(1+|y-x_1|)^{\frac{N+2}2+\tau}}.
\]
Thus, we have proved

\[
\|J_1\|_{**}\le C\bigl(
\frac k\mu\bigr)^{\frac{N+2}2-\tau} .
\]

Now, we estimate  $J_2$.  For $y\in \Omega_1$, and $j>1$, using
Lemma~\ref{laa1}, we have

\[
\begin{split}
&U_{x_j,\Lambda}^{2^*-1}(y)\le C\frac{1}{(1+|y-x_1|)^{\frac{N+2}2}}\frac{1}{(1+|y-x_j|)^{\frac{N+2}2}}\\
\le &C\frac{1}{(1+|y-x_1|)^{\frac{N+2}2+\tau}}\frac{1}{|x_1-x_j|^{\frac{N+2}2-\tau}},
\end{split}
\]
which implies

\begin{equation}\label{10-l2-5-3}
\begin{split}
&\Bigl|\sum_{j=2}^k \Bigl(
K\bigl(\frac {|y|}\mu\bigr)-1\Bigr)U_{x_j,\Lambda}^{2^*-1}\Bigr|\\
\le & C\frac{1}{(1+|y-x_1|)^{\frac{N+2}2+\tau}}\sum_{j=2}^k\frac{1}{|x_1-x_j|^{\frac{N+2}2-\tau}}\\
\le & C\frac{1}{(1+|y-x_1|)^{\frac{N+2}2+\tau}}\bigl(\frac k\mu\bigr)^{\frac{N+2}2-\tau}.
\end{split}
\end{equation}

For $y\in \Omega_1$ and $||y|-\mu r_0|\ge \delta \mu$, where
$\delta>0$ is a fixed constant, then

\[
||y|-|x_1||\ge ||y|-\mu r_0|-||x_1|-\mu r_0|\ge \frac12 \delta \mu.
\]
As a result,

\begin{equation}\label{11-l2-5-3}
\begin{split}
&\Bigl|U_{x_1,\Lambda}^{2^*-1}\Bigl(
K\bigl(\frac {|y|}\mu\bigr)-1\Bigr)\Bigr|\\
\le & C\frac{1}{(1+|y-x_1|)^{\frac{N+2}2+\tau}}\frac1{\mu^{\frac{N+2}2-\tau}}.
\end{split}
\end{equation}

If  $y\in \Omega_1$ and $||y|-\mu r_0|\le \delta \mu$, then

\[
\begin{split}
&\Bigl|K\bigl(\frac {|y|}\mu\bigr)-1\Bigr|\le C|\frac {|y|}\mu-r_0|^m\\
\le  & \frac{C}{\mu^m} \Bigl((||y|-|x_1||)^m+||x_1|-\mu r_0|)^{m}\Bigr)
\\
\le & \frac{C}{\mu^m} ||y|-|x_1||^m+\frac{C}{\mu^{m+\bar\theta}},
\end{split}
\]
and

\[
||y|-|x_1||\le ||y|-\mu r_0|+|\mu  r_0-|x_1||\le 2\delta \mu.
\]

 But

\[
\begin{split}
&
\frac{ ||y|-|x_1||^m }{\mu^{m}}\frac1{(1+|y-x_1|)^{N+2}}\\
= & \frac1{\mu^{\frac {N+2}2-\tau}}
\frac1{(1+|y-x_1|)^{\frac{N+2}2+\tau}}
\frac{ ||y|-|x_1||^m }{\mu^{ m -\frac {N+2}2+\tau}}\frac1{(1+|y-x_1|)^{\frac{N+2}2-\tau}}\\
\le & \frac C{\mu^{\frac{N+2}2-\tau}}
\frac1{(1+|y-x_1|)^{\frac{N+2}2+\tau}}
\frac{ ||y|-|x_1||^{\frac{N+2}2-\tau} }{(1+|y-x_1|)^{\frac{N+2}2-\tau}}\\
\le & \frac C{\mu^{\frac{N+2}2-\tau}}
\frac1{(1+|y-x_1|)^{\frac{N+2}2+\tau}}.
\end{split}
\]
 Thus, we obtain

\begin{equation}\label{13-l2-5-3}
\begin{split}
&\Bigl|U_{x_1,\Lambda}^{2^*-1}\Bigl(
K\bigl(\frac {|y|}\mu\bigr)-1\Bigr)\Bigr|\\
\le & \frac C{\mu^{\frac{N+2}2-\tau}}
\frac1{(1+|y-x_1|)^{\frac{N+2}2+\tau}},\quad ||y|-\mu r_0|\le \delta \mu.
\end{split}
\end{equation}
 Combining \eqref{10-l2-5-3}, \eqref{11-l2-5-3} and
 \eqref{13-l2-5-3}, we reach

\[
\|J_2\|_{**}\le \frac C{\mu^{\frac{N+2}2-\tau}}+C\bigl(\frac k\mu\bigr)^{\frac{N+2}2-\tau}\le
C\bigl(\frac k\mu\bigr)^{\frac{N+2}2-\tau} .
\]

\end{proof}

Now, we are ready to prove Proposition~\ref{p1-6-3}.

\begin{proof}[Proof of Proposition~\ref{p1-6-3}]

Let us recall that

\[
 \mu =
k^{\frac{N-2}{N-m-2}}.
\]

 Let

\[
E=\bigl\{ u:  u\in C(\R^N), \|u\|_*\le \bigl(\frac k\mu\bigr)^{\frac {N+2}2-\tau
-\eta}, \int_{\R^N} U_{x_i,\Lambda}^{2*-2} Z_{i,l}\phi=0,
\;i=1,\cdots,k,\; l=1,2
\bigr\},
\]
where $\eta>0$ is a fixed small constant.  Then, \eqref{re1} is
equivalent to

\[
\phi= A(\phi)=: L(N(\phi))+L(l_k).
\]
 We will prove that $A$ is a contraction map from $E$ to $E$.

In fact, if $N\ge 6$,

\begin{equation}\label{100-6-3}
\begin{split}
&\|\phi\|_{*} \le C\|N(\phi)\|_{**} +C\|l_k\|_{**}\\
\le & C k^{\frac{4}{N-2}} \|\phi\|_{*}^{2^*-1}
+C\bigl(\frac k\mu\bigr)^{\frac{N+2}2-\tau}\\
\le & C k^{\frac{4}{N-2}} \bigl(\frac k\mu\bigr)^{(\frac{N+2}2-\tau-\eta)\frac{N+2}{N-2}}
+C\bigl(\frac k\mu\bigr)^{\frac{N+2}2-\tau}\\
=& \frac { C k^{\frac{4}{N-2}+\frac{N+2}{N-2}(2-\eta)-\frac{4\tau}{N-2}
}}{\mu^{ \frac{N+2}{N-2}(2-\eta)-\frac{4\tau}{N-2}}}
\bigl(\frac k\mu\bigr)^{\frac{N+2}2-\tau}
\le  \bigl(\frac k\mu\bigr)^{\frac {N+2}2-\tau
-\eta},
\end{split}
\end{equation}
since

\[
\begin{split}
&\frac{N-2}{N-m-2}\Bigl(\frac{N+2}{N-2}(2-\eta)-\frac{4\tau}{N-2}\Bigr)\\
\ge &\frac{N-2}{N-4}\Bigl(\frac{N+2}{N-2}(2-\eta)-\frac{4\tau}{N-2}\Bigr)
>\frac{4}{N-2}+\frac{N+2}{N-2}(2-\eta)-\frac{4\tau}{N-2},
\end{split}
\]
if we take $\eta>0$ is small and $\tau$ is close to 1. Thus, $A$
maps $E$ to $E$.

On the other hand,
\[
\| A(\phi_1)-A_(\phi_2)\|_{*}= \|L(N(\phi_1))-L(N(\phi_2))\|_{*}
\le C \|N(\phi_1)-N(\phi_2)\|_{**}.
\]

If $N\ge 6$, then

\[
|N'(t)|\le C|t|^{2^*-2}.
\]
As a result,

\[
\begin{split}
&|N(\phi_1)-N(\phi_2)|\le C\bigl(|\phi_1|^{2^*-2}+|\phi_2|^{2^*-2}\bigr)|\phi_1-\phi_2|\\
\le & C\bigl(\|\phi_1\|_{*}^{2^*-2}+\|\phi_2\|_*^{2^*-2}\bigr)\|\phi_1-\phi_2\|_*\Bigl(\sum_{j=1}^k
\frac1{(1+|y-x_j|)^{\frac{N-2}2+\tau}}\Bigr)^{2^*-1}
\end{split}
\]

As before, we have

\[
\Bigl(\sum_{j=1}^k
\frac1{(1+|y-x_j|)^{\frac{N-2}2+\tau}}\Bigr)^{2^*-1}\\
\le C  k^{\frac{4}{N-2}}
 \sum_{j=1}^k
\frac1{(1+|y-x_j|)^{\frac{N+2}2+\tau}}.
\]
So,

\[
\begin{split}
&\| A(\phi_1)-A_(\phi_2)\|_{*}\le C
\|N(\phi_1)-N(\phi_2)\|_{**}\\
\le &
C k^{\frac{4}{N-2}}
\bigl(\|\phi_1\|_{*}^{2^*-2}+\|\phi_2\|_*^{2^*-2}\bigr)\|\phi_1-\phi_2\|_*\\
\le& \frac { C k^{\frac{4}{N-2}+\frac{2(N+2)}{N-2}-\frac{4(\tau+\eta)}{N-2}
}}{\mu^{ \frac{2(N+2)}{N-2}-\frac{4(\tau+\eta)}{N-2}}}
\|\phi_1-\phi_2\|_*\le \frac12  \|\phi_1-\phi_2\|_*.
\end{split}
\]
  Thus, $A$ is a contraction map.

The case $N=5$ can be discussed in a similar way.

It follows from the contraction mapping theorem that there is a
unique $\phi\in E$, such that

\[
\phi=A(\phi).
\]
Moreover, it follows from Proposition~\ref{p1} that

\[
\|\phi\|_*\le C\bigl( \frac k{\mu}\bigr)^{\frac{N+2}2-\tau}.
\]

\end{proof}

\section{Proof of Theorem \ref{th11}}

Let

\[
F(r,\Lambda)= I\bigl(  W_r +\phi\bigr),
\]
where
$
r=|x_1|,
$
$\phi$ is the function obtained in Proposition~\ref{p1-6-3}, and

\[
I(u)=\frac12\int_{\R^N} |Du|^2-\frac1{2^*}\int_{\R^N} K\bigl(\frac{|y|}\mu\bigr)|u|^{2^*}.
\]

\begin{proposition}\label{p2-6-3}
We have

\[
\begin{split}
F(r,\Lambda)= & I(  W_r)+O\Bigl(\frac k{\mu^{m+\sigma}}\Bigr)\\
=& k\Bigl( A +\frac{B_1}{\Lambda^{m}\mu^m}
+\frac{B_2}{\Lambda^{m-2}\mu^m}
(\mu r_0 -|x_1|))^2\\
&\quad-\sum_{i=2}^k\frac{ B_3 }{\Lambda^{N-2}|x_1-x_j|^{N-2}}+
O\Bigl(\frac1{\mu^{m+\sigma}}+\frac{1}{\mu^m}
|\mu r_0 -|x_1||^3\Bigr)
\Bigr),
\end{split}
\]
where $\sigma>0$ is a fixed constant, $B_i>0$, $i=1,2,3$, is some
constant.

\end{proposition}

\begin{proof}
Since

\[
\bigl\langle I'\bigl(  W_r \bigr),\phi\bigr\rangle=0,
\]
there is $t\in (0,1)$ such that

\[
\begin{split}
&F(r,\Lambda)= I(  W_r) +\frac12 D^2I\bigl(  W_r +t \phi\bigr) (\phi,\phi)\\
=& I(  W_r) +\int_{\R^N} \bigl(|D \phi|^2-(2^*-1)K\bigl(\frac {|y|}\mu\bigr)
\bigl( W_r+t \phi\bigr)^{2^*-2}\phi^2\bigr)\\
=& I(  W_r) + (2^*-1)\int_{\R^N}
K\bigl(\frac {|y|}\mu\bigr) \Bigl(\bigl( W_r +t \phi\bigr)^{2^*-2}-W_r^{2^*-2}
\Bigr)\phi^2\\
&+\int_{\R^N}\bigl( N(\phi)+l_k\bigr)\phi\\
=& I(  W_r)+O\Bigl(\int_{\R^N}\bigl(|\phi|^{2^*}+|N(\phi)||\phi|+|l_k||\phi|\bigr)\Bigr).
\end{split}
\]

But

\[
\begin{split}
&\int_{\R^N}\bigl(|N(\phi)||\phi|+|l_k||\phi|\bigr)\\
\le  &C\Bigl(\|N(\phi)\|_{**}+\|l_k\|_{**}\Bigr)\|\phi\|_*
\int_{\R^N} \sum_{j=1}^k\frac1{(1+|y-x_j|)^{\frac{N+2}2+\tau}}\sum_{i=1}^k\frac1{(1+|y-x_i|)^{\frac{N-2}2+\tau}}.
\end{split}
\]
Using Lemma~\ref{laa1}

\[
\begin{split}
&\sum_{j=1}^k\frac1{(1+|y-x_j|)^{\frac{N+2}2+\tau}}\sum_{i=1}^k\frac1{(1+|y-x_i|)^{\frac{N-2}2+\tau}}\\
=&\sum_{j=1}^k\frac1{(1+|y-x_j|)^{N+2\tau}}+\sum_{j=1}^k\sum_{i\ne j}\frac1{(1+|y-x_j|)^{\frac{N+2}2+\tau}}
\frac1{(1+|y-x_i|)^{\frac{N-2}2+\tau}}\\
\le &\sum_{j=1}^k\frac1{(1+|y-x_j|)^{N+2\tau}}+C \sum_{j=1}^k \frac1{(1+|y-x_j|)^{N+\tau}}\sum_{j=2}^k
\frac1{|x_j-x_1|^\tau}\\
\le C & \sum_{j=1}^k\frac1{(1+|y-x_j|)^{N+\tau}},
\end{split}
\]
since  $\tau>1$. Thus, we obtain

\[
\int_{\R^N}\bigl(|N(\phi)||\phi|+|l_k||\phi|\bigr)
\le  C k\Bigl(\|N(\phi)\|_{**}+\|l_k\|_{**}\Bigr)\|\phi\|_*\le Ck \bigl(\frac{k}{\mu}\bigr)^{N+2-2\tau}.
\]

On the other hand,

\[
\int_{\R^N} |\phi|^{2^*}\le C \|\phi\|_*^{2^*}
\int_{\R^N} \Bigl(\sum_{j=1}^k\frac1{(1+|y-x_j|)^{\frac{N-2}2+\tau}}\Bigr)^{2^*}.
\]
But using Lemma~\ref{laa1}, if $y\in\Omega_1$,

\[
\begin{split}
&
\sum_{j=2}^k\frac1{(1+|y-x_j|)^{\frac{N-2}2+\tau}}\\
\le &\sum_{j=2}^k\frac1{(1+|y-x_1|)^{\frac{N-2}4+\frac12\tau}}\frac1{(1+|y-x_j|)^{\frac{N-2}4+\frac12\tau}}\\
\le & C\frac1{(1+|y-x_1|)^{\frac{N-2}2+\frac12\bar\eta }}
\sum_{j=2}^k \frac1{|x_j-x_1|^{ \tau-\frac12\bar\eta}}\le C\frac1{(1+|y-x_1|)^{\frac{N-2}2+\frac12\bar\eta}},
\end{split}
\]
  Thus,

\[
\Bigl(\sum_{j=1}^k\frac1{(1+|y-x_j|)^{\frac{N-2}2+\tau}}\Bigr)^{2^*}\le \frac C{(1+|y-x_1|)^{N+2^*\frac12 \bar\eta}},
\quad y\in\Omega_1.
\]
Thus,

\[
\int_{\R^N} \Bigl(\sum_{j=1}^k\frac1{(1+|y-x_j|)^{\frac{N-2}2+\tau}}\Bigr)^{2^*}\le Ck.
\]
So, we have proved

\[
\int_{\R^N} |\phi|^{2^*}\le C k\|\phi\|_*^{2^*}\le Ck\bigl( \frac{k}{\mu}\bigr)^{2^*(\frac{N+2}2-\tau) }.
\]

\end{proof}

\begin{proposition}\label{p1-7-3}
We have

\[
\begin{split}
&\frac{\partial F(r,\Lambda)}{\partial \Lambda}\\
=& k\Bigl( -\frac{B_1 m }{\Lambda^{m+1}\mu^m}
+\sum_{i=2}^k\frac{ B_3(N-2) }{\Lambda^{N-1}|x_1-x_j|^{N-2}}
+
O\Bigl(\frac1{\mu^{m+\sigma}}+\frac{1}{\mu^m}
|\mu r_0 -|x_1||^2\Bigr)
\Bigr),
\end{split}
\]
where $\sigma>0$ is a fixed constant.

\end{proposition}

\begin{proof}
 We have

\[
\begin{split}
&\frac{\partial F(r,\Lambda)}{\partial \Lambda}
= \bigl\langle I'(W_r+\phi), \frac{\partial W_r}{\partial \Lambda}+\frac{\partial \phi}{\partial \Lambda}
\bigr\rangle\\
=&\bigl\langle I'(W_r+\phi), \frac{\partial W_r}{\partial \Lambda}
\bigr\rangle +\sum_{l=1}^2\sum_{i=1}^kc_l\bigl\langle U^{2^*-2}_{x_i,\Lambda} Z_{i,l},
\frac{\partial \phi}{\partial \Lambda}
\bigr\rangle.
\end{split}
\]
But

\[
\bigl\langle U^{2^*-2}_{x_i,\Lambda} Z_{i,l}, \frac{\partial \phi}{\partial \Lambda}
\bigr\rangle=-
\bigl\langle \frac{\partial (U^{2^*-2}_{x_{i},\Lambda} Z_{i,l})}{\partial \Lambda},  \phi
\bigr\rangle
\]
Thus, using Proposition~\ref{p1-6-3},

\[
\begin{split}
&\Bigl|\sum_{i=1}^kc_l\bigl\langle U^{2^*-2}_{x_i,\Lambda} Z_{i,l}, \frac{\partial \phi}{\partial \Lambda}
\bigr\rangle\Bigr|\\
\le &C|c_l|\|\phi\|_* \int_{\R^N} \sum_{i=1}^k\frac1{(1+|y-x_i|)^{N+2}}\sum_{j=1}^k
\frac1{(1+|y-x_j|)^{\frac{N-2}2+\tau}}\\
\le & \frac C{\mu^{m+\sigma}}.
\end{split}
\]

On the other hand,

\[
\int_{\R^N} D (W_r+\phi) D\frac{\partial W_r}{\partial \Lambda}=
\int_{\R^N} D W_r D\frac{\partial W_r}{\partial \Lambda},
\]
and

\[
\begin{split}
&\int_{\R^N} K\bigl(\frac {|y|}\mu\bigr) (W_r+\phi)^{2^*-1}\frac{\partial W_r}{\partial \Lambda}\\
=&\int_{\R^N} K\bigl(\frac {|y|}\mu\bigr) W_r^{2^*-1}\frac{\partial W_r}{\partial \Lambda}+
(2^*-1)\int_{\R^N} K\bigl(\frac {|y|}\mu\bigr) W_r^{2^*-2}\frac{\partial W_r}{\partial \Lambda}\phi+
O\Bigl(\int_{\R^N} |\phi|^{2^*}\Bigr).
\end{split}
\]
Moreover, from $\phi \in E$,

\[
\begin{split}
&\int_{\R^N} K\bigl(\frac {|y|}\mu\bigr) W_r^{2^*-2}\frac{\partial W_r}{\partial \Lambda}\phi\\
=&\int_{\R^N} K\bigl(\frac {|y|}\mu\bigr) \bigl(W_r^{2^*-2}\frac{\partial W_r}{\partial \Lambda}-
\sum_{j=1}^k U_{x_j,\Lambda}^{2^*-2}\frac{\partial U_{x_j,\Lambda} }{\partial \Lambda}\bigr) \phi
+\sum_{j=1}^k \int_{\R^N} \bigl(K\bigl(\frac {|y|}\mu\bigr)-1\bigr)
 U_{x_j,\Lambda}^{2^*-2}\frac{\partial U_{x_j,\Lambda} }{\partial \Lambda} \phi\\
 =& k \int_{\Omega_1} K\bigl(\frac {|y|}\mu\bigr) \bigl(W_r^{2^*-2}\frac{\partial W_r}{\partial \Lambda}-
\sum_{j=1}^k U_{x_j,\Lambda}^{2^*-2}\frac{\partial U_{x_j,\Lambda} }{\partial \Lambda}\bigr) \phi+
k\int_{\R^N} \bigl(K\bigl(\frac {|y|}\mu\bigr)-1\bigr)
 U_{x_1,\Lambda}^{2^*-2}\frac{\partial U_{x_1,\Lambda} }{\partial \Lambda} \phi,
\end{split}
\]

\[
\begin{split}
&\Bigl|\int_{\Omega_1} K\bigl(\frac {|y|}\mu\bigr) \bigl(W_r^{2^*-2}\frac{\partial W_r}{\partial \Lambda}-
\sum_{j=1}^k U_{x_j,\Lambda}^{2^*-2}\frac{\partial U_{x_j,\Lambda} }{\partial \Lambda}\bigr) \phi\Bigr|
\\
\le  & C \int_{\Omega_1}  \Bigl(U_{x_1,\Lambda}^{2^*-2}\sum_{j=2}^k U_{x_j,\Lambda}+
\sum_{j=2}^k U_{x_j,\Lambda}^{2^*-1}\Bigr)|\phi|\\
\le &\frac{C}{\mu^{m+\sigma}},
\end{split}
\]
and

\[
\begin{split}
&\Bigl|\int_{\R^N} \bigl(K\bigl(\frac {|y|}\mu\bigr)-1\bigr)
 U_{x_1,\Lambda}^{2^*-2}\frac{\partial U_{x_1,\Lambda} }{\partial \Lambda} \phi\Bigr|\\
 \le & \Bigl|\int_{||y|-\mu r_0|\le \sqrt\mu} \bigl(K\bigl(\frac {|y|}\mu\bigr)-1\bigr)
 U_{x_1,\Lambda}^{2^*-2}\frac{\partial U_{x_1,\Lambda} }{\partial \Lambda} \phi\Bigr|+
 \Bigl|\int_{||y|-\mu r_0|\ge \sqrt\mu} \bigl(K\bigl(\frac {|y|}\mu\bigr)-1\bigr)
 U_{x_1,\Lambda}^{2^*-2}\frac{\partial U_{x_1,\Lambda} }{\partial \Lambda} \phi\Bigr|\\
\le &\frac{C}{\mu^{m+\sigma}}.
\end{split}
\]
Thus, we have proved

\[
\frac{\partial F(r,\Lambda)}{\partial \Lambda}=\frac{\partial I(W_r)}{\partial \Lambda}+O\Bigl(\frac{1}{\mu^{m+\sigma}}
\Bigr),
\]
and the result follows from Proposition~\ref{pa3}.
\end{proof}

Since

\[
|x_j-x_1|=2|x_1|\sin\frac{(j-1)\pi}k, \quad j=2,\dots, k,
\]
we have

\[
\begin{split}
&\sum_{j=2}^k \frac1{|x_j-x_1|^{N-2}}=\frac1{(2|x_1|)^{N-2}}\sum_{j=2}^k \frac1{(\sin\frac{(j-1)\pi}k)^{N-2}}\\
=&
\begin{cases}
\frac1{(2|x_1|)^{N-2}}\sum_{j=2}^{\frac k2} \frac1{(\sin\frac{(j-1)\pi}k)^{N-2}}
+\frac1{(2|x_1|)^{N-2}},
& \text{if $k$ is even};\\
\frac1{(2|x_1|)^{N-2}}\sum_{j=2}^{[\frac k2]} \frac1{(\sin\frac{(j-1)\pi}k)^{N-2}}, &  \text{if $k$ is old}.
\end{cases}
\end{split}
\]
But

\[
0<c'\le \frac{\sin\frac{(j-1)\pi}k}{\frac{(j-1)\pi}k}\le c'',  \quad j=2,\cdots, [\frac k2].
\]
So, there is a constant $B_4>0$, such that

\[
\sum_{j=2}^k \frac1{|x_j-x_1|^{N-2}}=\frac{B_4k^{N-2}}{|x_1|^{N-2}}+O\Bigl(\frac k{|x_1|^{N-2}}\Bigr).
\]
Thus, we obtain

\[
\begin{split}
F(r,\Lambda)
=& k\Bigl( A +\frac{B_1}{\Lambda^{m}\mu^m}
+\frac{B_2}{\Lambda^{m-2}\mu^m}
(\mu r_0 -r)^2\\
&\quad-\frac{ B_4 k^{N-2}}{\Lambda^{N-2}r^{N-2}}+
O\Bigl(\frac1{\mu^{m+\sigma}}+\frac{1}{\mu^m}
|\mu r_0 -r|^3+\frac k{r^{N-2}}\Bigr)
\Bigr),
\end{split}
\]
and

\[
\begin{split}
&\frac{\partial F(r,\Lambda)}{\partial \Lambda}\\
=& k\Bigl( -\frac{B_1 m }{\Lambda^{m+1}\mu^m}
+\frac{ B_4(N-2) k^{N-2}}{\Lambda^{N-1}r^{N-2}}
+
O\Bigl(\frac1{\mu^{m+\sigma}}+\frac{1}{\mu^m}
|\mu r_0 -r|^2+\frac k{r^{N-2}}\Bigr)
\Bigr).
\end{split}
\]

 Let $\Lambda_0$ be the solution of

\[
-\frac{B_1 m }{\Lambda^{m+1}}
+\frac{ B_4(N-2) }{\Lambda^{N-1}r_0^{N-2}}=0.
\]
Then

\[
\Lambda_0= \Bigl(\frac {B_4(N-2)}{B_1 m r_0^{N-2}}\Bigr)^{\frac1{N-2-m}}.
\]

Define

\[
D=\bigl\{ (r,\Lambda):  r\in [\mu r_0-\frac1{\mu^{\bar\theta}},\mu r_0+\frac1{\mu^{\bar\theta}}],\;\;
\Lambda\in [\Lambda_0 -\frac1{\mu^{\frac32\bar\theta}},\Lambda_0 +\frac1{\mu^{\frac32\bar\theta}}]
\bigr\},
\]
where $\bar\theta>0$ is a small constant.

For any $(r,\Lambda)\in D$, we have

\[
\frac r\mu=r_0+O\bigl(\frac1{\mu^{1+\bar\theta}}\bigr).
\]
Thus,

\[
r^{N-2}=\mu^{N-2} \bigl( r_0^{N-2}+O\Bigl(\frac1{\mu^{1+\bar\theta}}\bigr)\Bigr).
\]
So,

\begin{equation}\label{1-8-3}
\begin{split}
F(r,\Lambda)
=& k\Bigl( A +\bigl(\frac{B_1}{\Lambda^{m}}-\frac{ B_4 }{\Lambda^{N-2}r_0^{N-2}}
\bigr)\frac1{\mu^m}\\
&\qquad+\frac{B_2}{\Lambda^{m-2}\mu^m}
(\mu r_0 -r)^2+
O\Bigl(\frac1{\mu^{m+\sigma}}+\frac{1}{\mu^m}
|\mu r_0 -r|^3+\frac k{\mu^{N-2}}\Bigr)
\Bigr),\quad (r,\Lambda)\in D,
\end{split}
\end{equation}
and

\begin{equation}\label{2-8-3}
\begin{split}
&\frac{\partial F(r,\Lambda)}{\partial \Lambda}\\
=& k\Bigl( \bigl(-\frac{B_1 m }{\Lambda^{m+1}}
+\frac{ B_4(N-2) }{\Lambda^{N-1}r_0^{N-2}}\bigr)\frac1{\mu^m}
+
O\Bigl(\frac1{\mu^{m+\sigma}}+\frac{1}{\mu^m}
|\mu r_0 -r|^2+\frac k{\mu^{N-2}}\Bigr)
\Bigr),\quad (r,\Lambda)\in D.
\end{split}
\end{equation}

Now, we define

\[
\bar F(r,\Lambda)=-F(r,\Lambda), \quad (r,\Lambda)\in D.
\]

Let

\[
\alpha_2= k(-A+\eta),\quad, \alpha_1=k\Bigl( -A -\bigl(\frac{B_1}{\Lambda_0^{m}}-\frac{ B_4 }{\Lambda_0^{N-2}r_0^{N-2}}
\bigr)\frac1{\mu^m}-\frac1{\mu^{m+\frac52\bar\theta}}\Bigr),
\]
where $\eta>0$ is a small constant.

Let

\[
\bar F^\alpha=\bigl\{ (r,\Lambda)\in D,  \bar F(r,\Lambda)\le \alpha\bigr\}.
\]

Consider

\[
\begin{cases}
\frac{ d r}{d t}=- D_r\bar F, &t>0;\\
\frac{ d \Lambda}{d t}=- D_\Lambda\bar F, &t>0;\\
(r,\Lambda)\in F^{\alpha_2}.
\end{cases}
\]
Then

\begin{proposition}\label{p1-8-3}
The flow $(r(t),\Lambda(t))$ does not leave $D$ before it reaches
$F^{\alpha_1}$.

\end{proposition}

\begin{proof}

If $\Lambda=\Lambda_0+\frac1{\mu^{\frac32\bar\theta}}$, noting that
$|r-\mu r_0|\le \frac1{\mu^{\bar\theta}}$, we obtain from
\eqref{2-8-3} that

\[
\frac{\partial \bar F(r,\Lambda)}{\partial \Lambda}
= k\Bigl( c'\frac1{\mu^{m+\frac 32 \bar\theta}}
+
O\Bigl(\frac1{\mu^{m+2 \bar\theta}}\Bigr)
\Bigr)>0.
\]
So, the flow does not leave $D$.

Similarly, if $\Lambda=\Lambda_0-\frac1{\mu^{\frac32\bar\theta}}$,
then  we obtain from \eqref{2-8-3} that

\[
\frac{\partial \bar F(r,\Lambda)}{\partial \Lambda}
= k\Bigl( -c'\frac1{\mu^{m+\frac 32 \bar\theta}}
+
O\Bigl(\frac1{\mu^{m+2 \bar\theta}}\Bigr)
\Bigr)<0.
\]
So, the flow does not leave $D$.

Suppose now  $|r-\mu r_0|=\frac1{\mu^{\bar\theta}}$. Since
$|\Lambda-\Lambda_0|\le \frac1{\mu^{\frac32\bar\theta}}$, we see

\[
\begin{split}
&\frac{B_1}{\Lambda^{m}}-\frac{ B_4 }{\Lambda^{N-2}r_0^{N-2}}
=\frac{B_1}{\Lambda_0^{m}}-\frac{ B_4 }{\Lambda_0^{N-2}r_0^{N-2}}+O\bigl(|\Lambda-\Lambda_0|^2\bigr)\\
=&\frac{B_1}{\Lambda_0^{m}}-\frac{ B_4 }{\Lambda_0^{N-2}r_0^{N-2}}+O\bigl(\frac1{\mu^{3\bar\theta}}\bigr).
\end{split}
\]
So, using \eqref{1-8-3}, we obtain

\begin{equation}\label{5-8-3}
\begin{split}
&\bar F(r,\Lambda)\\
= &k\Bigl(- A -\bigl(\frac{B_1}{\Lambda_0^{m}}-\frac{ B_4 }{\Lambda_0^{N-2}r_0^{N-2}}
\bigr)\frac1{\mu^m}
-\frac{B_2}{\Lambda_0^{m-2}\mu^m}
(\mu r_0 -r)^2+
O\Bigl(\frac1{\mu^{m+3\bar\theta}}
\Bigr)\Bigr)\\
\le &k\Bigl(- A -\bigl(\frac{B_1}{\Lambda_0^{m}}-\frac{ B_4 }{\Lambda_0^{N-2}r_0^{N-2}}
\bigr)\frac1{\mu^m}
-\frac{B_2}{\Lambda_0^{m-2}\mu^{m+2\bar\theta}}
+
O\Bigl(\frac1{\mu^{m+3\bar\theta}}
\Bigr)\Bigr)<\alpha_1.
\end{split}
\end{equation}

\end{proof}

\begin{proof}[Proof of Theorem \ref{th11}]

We will prove that $\bar F$, and thus $F$, has a critical point in
$D$.

 Define

\[
\begin{split}
\Gamma=\bigl\{ h: &h (r,\Lambda)=(h_1 (r,\Lambda), h_2 (r,\Lambda))\in D, (r,\Lambda)\in D\\
&h (r,\Lambda)= (r,\Lambda), \;\text{if}\;  |r-\mu r_0|=\frac1{\mu^{\bar\theta}}\bigr\}.
\end{split}
\]

Let

\[
c=\inf_{h\in\Gamma}\max_{(r,\Lambda)\in D} \bar F(h(r,\Lambda)).
\]
We claim that $c$ is a critical value of $\bar F$. To prove this, we
need to prove

\begin{itemize}

\item[(i)]  $\alpha_1 < c<\alpha_2$;

\item[(ii)]  $\sup_{|r-\mu r_0|=\frac1{\mu^{\bar\theta}}} \bar
F(h(r,\Lambda))<\alpha_1,\;\forall\; h\in \Gamma.$

\end{itemize}

To prove (ii),  let $h\in \Gamma$. Then for any  $\bar r$ with
$|\bar r-\mu r_0|=\frac1{\mu^{\bar\theta}}$, we have $h(\bar
r,\Lambda)= (\bar r,\tilde \Lambda)$ for some $\tilde \Lambda$.
Thus, by \eqref{5-8-3},

\[
\bar F(h(r,\Lambda))= \bar F(\bar r,\tilde \Lambda)<\alpha_1.
\]

Now we prove (i).  It is easy to see that

\[
 c<\alpha_2.
 \]

For any $h=(h_1,h_2)\in \Gamma$. Then $h_1(r,\Lambda)= r$, if
$|r-\mu r_0|=\frac1{\mu^{\bar\theta}}$. Define

\[
\tilde h_1(r)=h_1(r,\Lambda_0).
\]
Then  $\tilde h_1(r)= r$, if $|r-\mu r_0|=\frac1{\mu^{\bar\theta}}$.
So, there is a $\bar r\in (\mu r_0-\frac1{\mu^{\bar\theta}}, \mu
r_0+\frac1{\mu^{\bar\theta}})$, such that

\[
\tilde h_1(\bar r)=\mu r_0.
\]
Let $\bar \Lambda= h_2(\bar r,\Lambda_0)$. Then from \eqref{1-8-3}

\[
\begin{split}
&\max_{(r,\Lambda)\in D}\bar F(h(r,\Lambda))\ge \bar F(h(\bar r,\Lambda_0))=
\bar F(\mu r_0,\bar \Lambda)\\
=& k\Bigl(- A -\bigl(\frac{B_1}{\bar \Lambda^{m}}-\frac{ B_4 }{\bar\Lambda^{N-2}r_0^{N-2}}
\bigr)\frac1{\mu^m}
+
O\Bigl(\frac1{\mu^{m+\sigma}}+\frac k{\mu^{N-2}}\Bigr)
\Bigr)\\
=& k\Bigl(- A -\bigl(\frac{B_1}{\Lambda_0^{m}}-\frac{ B_4 }{\Lambda_0^{N-2}r_0^{N-2}}
\bigr)\frac1{\mu^m}
+
O\Bigl(\frac1{\mu^{m+3\bar\theta}}\Bigr)
\Bigr)>\alpha_1.
\end{split}
\]

\end{proof}

\appendix

\section{Energy Expansion}

In all of the appendixes, we always assume that

\[
x_j=\bigl(r \cos\frac{2(j-1)\pi}k, r\sin\frac{2(j-1)\pi}k,0\bigr),\quad j=1,\cdots,k,
\]
where $0$ is the zero vector in $\R^{N-2}$, and $r\in [r_0 \mu
-\frac1{\mu^{\bar\theta}},r_0 \mu+\frac1{\mu^{\bar\theta}}]$ for
some small $\bar\theta>0$.

Let recall that

\[
\mu = k^{\frac{N-2}{N-2-m}},
\]

\[
I(u)=\frac12\int_{\R^N} |Du|^2-\frac1{2^*}\int_{\R^N} K\bigl(\frac{|y|}\mu\bigr)|u|^{2^*},
\]

\[
U_{x_j,\Lambda}(y)=\bigl(N(N-2)\bigr)^{\frac{N-2}4}
\frac{\Lambda^{\frac{N-2}2}}{(1+\Lambda^2|y-x_j|^2)^{\frac{N-2}2}},
\]
and

\[
W_r(y)=\bigl(N(N-2)\bigr)^{\frac{N-2}4}\sum_{j=1}^k \frac{\Lambda^{\frac{N-2}2}}{(1+\Lambda^2|y-x_j|^2)^{\frac{N-2}2}}.
\]

In this section, we will calculate $I(W_r)$.

\begin{proposition}\label{pa2}
 We have

\[
\begin{split}
I(W_r)=& k\Bigl( A +\frac{B_1}{\Lambda^{m}\mu^m}
+\frac{B_2}{\Lambda^{m-2}\mu^m}
(\mu r_0 -r))^2\\
&\quad-\sum_{i=2}^k\frac{ B_3 }{\Lambda^{N-2}|x_1-x_j|^{N-2}}+
O\Bigl(\frac1{\mu^{m+\sigma}}+\frac{1}{\mu^m}
|\mu r_0 -r|^3\Bigr)
\Bigr),
\end{split}
\]
where $B_i$, $i=1,2,3$, is some positive constant,  $A>0$ is a
constant, and $r=|x_1|$.

\end{proposition}

\begin{proof}

By using the symmetry, we have

\[
\begin{split}
&\int_{\R^N } |D W_r|^2= \sum_{j=1}^k \sum_{i=1}^k \int_{\R^N } U_{x_j,\Lambda}^{2^*-1}U_{x_i,\Lambda}\\
=& k \Bigl(\int_{\R^N} U^{2^*}_{0,1} +\sum_{i=2}^k \int_{\R^N } U_{x_1,\Lambda}^{2^*-1}U_{x_i,\Lambda}\Bigr)\\
=& k \Bigl(\int_{\R^N} U^{2^*}_{0,1} +\sum_{i=2}^k\frac{ B_0 }{\Lambda^{N-2}
|x_1-x_j|^{N-2}}+O\Bigl(
\sum_{i=2}^k\frac{ 1 }{|x_1-x_j|^{N-2+\sigma}}\Bigr)\Bigr).
\end{split}
\]

Let

\[
\Omega_{j}=\bigl\{ y:\; y=(y',y'')=\R^2\times \R^{N-2},\;
  \bigl\langle \frac {y'}{|y'|}, \frac{x_j}{|x_j|}\bigr\rangle \ge \cos\frac {\pi}k\bigr\}.
\]
Then,

\[
\begin{split}
&
\int_{\R^N } K\bigl(\frac {|y|}\mu\bigr) |W_r|^{2^*}= k\int_{\Omega_1}K\bigl(\frac {|y|}\mu\bigr) |W_r|^{2^*}\\
=& k\Bigl(\int_{\Omega_1}K\bigl(\frac {|y|}\mu\bigr) U_{x_1,\Lambda}^{2^*}-2^*
\int_{\Omega_1}K\bigl(\frac {|y|}\mu\bigr)
\sum_{i=2}^k U_{x_1,\Lambda}^{2^*-1}U_{x_i,\Lambda}\\
&\qquad+
O\Bigl(
\int_{\Omega_1} U_{x_1,\Lambda}^{2^*/2}\bigl(
\sum_{i=2}^k U_{x_i,\Lambda}\bigr)^{2^*/2}\Bigr)\Bigr).
\end{split}
\]

Note that for $y\in \Omega_1$, $|y-x_i|\ge |y-x_1|$. Using
Lemma~\ref{laa1}, we find

\[
\begin{split}
&
\sum_{i=2}^k U_{x_i,\Lambda}\le C\sum_{i=2}^k\frac1{(1+|y-x_1|)^{\frac{N-2}2}}\frac1{(1+|y-x_i|)^{\frac{N-2}2}}\\
\le & \frac1{(1+|y-x_1|)^{N-2-\alpha}}\sum_{i=2}^k \frac1{|x_i-x_1|^\alpha}.
\end{split}
\]
If we take  the constant $\alpha$ with  $\max (1,\frac
{(N-2)^2}N)<\alpha <N-2$, then

\[
\int_{\Omega_1} U_{x_1,\Lambda}^{2^*/2}\bigl(
\sum_{i=2}^k U_{x_i,\Lambda}\bigr)^{2^*/2}= O\Bigl( \bigl(\frac k \mu\bigr)^{N-2+\sigma}\Bigr).
\]

On the other hand, it is easy to show

\[
\begin{split}
&\int_{\Omega_1}K\bigl(\frac {|y|}\mu\bigr)
\sum_{i=2}^k U_{x_1,\Lambda}^{2^*-1}U_{x_i,\Lambda}\\
=& \int_{\Omega_1}
\sum_{i=2}^k U_{x_1,\Lambda}^{2^*-1}U_{x_i,\Lambda}
+\int_{\Omega_1}\Bigl(K\bigl(\frac {|y|}\mu\bigr)-1\Bigr)
\sum_{i=2}^k U_{x_1,\Lambda}^{2^*-1}U_{x_i,\Lambda}\\
=&\sum_{i=2}^k\frac{ B_0 }{\Lambda^{N-2}|x_1-x_j|^{N-2}}+O\Bigl( \bigl(\frac k \mu\bigr)^{N-2+\sigma}\Bigr).
\end{split}
\]

Finally,

\[
\begin{split}
&\int_{\Omega_1}K\bigl(\frac {|y|}\mu\bigr) U_{x_1,\Lambda}^{2^*}\\
=&\int_{\R^N}  U_{0,1}^{2^*}-\frac{c_0 }{\mu^m}\int_{\Omega_1}||y|-\mu r_0|^mU_{x_1,\Lambda}^{2^*}\\
&+
O\Bigl(\mu^{-m-\theta}\int_{\Omega_1}||y|-\mu r_0|^{m+\theta}U_{x_1,\Lambda}^{2^*}\Bigr)\\
=&\int_{\R^N}  U_{0,1}^{2^*}-\frac{c_0 }{\mu^m}\int_{\Omega_1}||y|-\mu r_0|^mU_{x_1,\Lambda}^{2^*}
+
O\Bigl(\frac1{\mu^{m+\theta}}\Bigr)\\
=&\int_{\R^N}  U_{0,1}^{2^*}-\frac{c_0 }{\mu^m}\int_{\R^N}||y-x_1|-\mu r_0|^mU_{0,\Lambda}^{2^*}
+
O\Bigl(\frac1{\mu^{m+\theta}}\Bigr).
\end{split}
\]
But

\[
\begin{split}
&
||y-x_1|-\mu r_0|^m= | |x_1|-y_1 +O(\frac1{|x_1|})-\mu r_0|^m\\
=& |y_1|^m + m| y_1|^{m-2} y_1 (\mu r_0 -|x_1|+O(\frac1{|x_1|})) \\
&+ \frac12 m(m-1) | y_1|^{m-2}
(\mu r_0 -|x_1|+O(\frac1{|x_1|}))^2+ O\Bigl( (\mu r_0 -|x_1|+O(\frac1{|x_1|}))^{2+\sigma}\Bigr)
\end{split}
\]
Thus,

\[
\begin{split}
&\int_{\R^N}||y-x_1|-\mu r_0|^mU_{0,\Lambda}^{2^*}\\
=&\int_{\R^N}|y_1|^mU_{0,\Lambda}^{2^*}+\frac12 m(m-1) \int_{\R^N}|y_1|^{m-2}U_{0,\Lambda}^{2^*}
(\mu r_0 -|x_1|))^2\\
&+O\Bigl(|\mu r_0 -|x_1||^{2+\sigma}\Bigr).
\end{split}
\]
Thus, we have proved

\
\[
\begin{split}
&
\int_{\R^N } K\bigl(\frac {|y|}\mu\bigr) |W_r|^{2^*}\\
=& k\Bigl(\int_{\R^N}  |U_{0,1}|^{2^*} -\frac{c_0}{\Lambda^{m}\mu^m}
\int_{\R^N}|y_1|^mU_{0,1}^{2^*}\\
&\quad-\frac{c_0}{\Lambda^{m-2}\mu^m}
\frac12 m(m-1) \int_{\R^N}|y_1|^{m-2}U_{0,1}^{2^*}
(\mu r_0 -|x_1|))^2\\
&\quad+2^*\sum_{i=2}^k\frac{ B_0 }{\Lambda^{N-2}|x_1-x_j|^{N-2}}+
O\Bigl(\frac1{\mu^{m+\sigma}}\Bigr)
\Bigr).
\end{split}
\]

\end{proof}

We also need to calculate  $ \frac{\partial I(W_r)}{\partial
\Lambda}$.

\begin{proposition}\label{pa3}
 We have

\[
\begin{split}
\frac{\partial I(W_r)}{\partial
\Lambda}=& k\Bigl( -\frac{m B_1}{\Lambda^{m+1}\mu^m}
+\sum_{i=2}^k\frac{ B_3 (N-2) }{\Lambda^{N-1}|x_1-x_j|^{N-2}}\\
&\qquad+
O\Bigl(\frac1{\mu^{m+\sigma}}+\frac{1}{\mu^m} |\mu r_0
-|x_1||^2\Bigr) \Bigr),
\end{split}
\]
where $B_i$, $i=1,2,3$, is same positive constant in
Proposition~\ref{pa2}

\end{proposition}

\begin{proof}

The proof of this proposition is similar to the proof of
Proposition~\ref{pa2}. So we just sketch it.

We have

\[
\begin{split}
\frac{\partial I(W_r)}{\partial
\Lambda}
=& k\Bigl( (2^*-1) \sum_{i=2}^k \int_{\R^N}  U_{x_1,\Lambda}^{2^*-2}
\frac{\partial U_{x_1,\Lambda}}{\partial
\Lambda}U_{x_i,\Lambda}\\
&\qquad-\int_{\Omega_1} K\bigl(\frac {|y|}\mu\bigr) W_r^{2^*-1}
\frac{\partial W_r }{\partial
\Lambda}
\Bigr).
\end{split}
\]

It is easy to check that for $y\in \Omega_1$,

\[
\Bigl| \frac{\partial }{\partial
\Lambda}\Bigl( W_r^{2^*}- U_{x_1,\Lambda}^{2^*}- 2^* U_{x_1,\Lambda}^{2^*-1}\sum_{i=2}^k U_{x_i,\Lambda}\Bigr)
\Bigr|
\le  C U_{x_1,\Lambda}^{2^*/2}\bigl(\sum_{i=2}^k U_{x_i,\Lambda}\bigr)^{2^*/2}.
\]
Thus,

\[
 \frac{\partial }{\partial
\Lambda} W_r^{2^*}= \frac{\partial }{\partial
\Lambda}U_{x_1,\Lambda}^{2^*}+2^* \frac{\partial }{\partial
\Lambda}\bigl(U_{x_1,\Lambda}^{2^*-1}\sum_{i=2}^k U_{x_i,\Lambda}\bigr)
+
O\Bigl(U_{x_1,\Lambda}^{2^*/2}\bigl(\sum_{i=2}^k U_{x_i,\Lambda}\bigr)^{2^*/2}\Bigr).
\]
As a result, we have
\[
\begin{split}
 &2^* \int_{\Omega_1}K\bigl(\frac {|y|}\mu\bigr) W_r^{2^*-1}
\frac{\partial W_r }{\partial
\Lambda}\\
=& \int_{\Omega_1}K\bigl(\frac {|y|}\mu\bigr)\frac{\partial }{\partial
\Lambda}U_{x_1,\Lambda}^{2^*}+
2^* \int_{\Omega_1}K\bigl(\frac {|y|}\mu\bigr)
\frac{\partial }{\partial
\Lambda}\bigl(U_{x_1,\Lambda}^{2^*-1}\sum_{i=2}^k U_{x_i,\Lambda}\bigr)
\\
&+O\Bigl(\int_{\Omega_1}U_{x_1,\Lambda}^{2^*/2}\bigl(\sum_{i=2}^k U_{x_i,\Lambda}\bigr)^{2^*/2}\Bigr).
\end{split}
\]
So, we obtain the desired result.

\end{proof}

\section{Basic Estimates}

For each fixed $i$ and $j$, $i\ne j$, consider the following
function

\begin{equation}\label{aa1}
g_{ij}(y)= \frac{1}{(1+|y-x_j|)^{\alpha}}\frac{1}{(1+|y-x_i|)^{\beta}},
\end{equation}
where $\alpha\ge 1$ and $\beta\ge 1$ are two constants.

\begin{lemma}\label{laa1}
For any constant $0<\sigma\le \min(\alpha,\beta)$, there is a
constant $C>0$, such that

\[
g_{ij}(y)\le \frac{C}{|x_i-x_j|^\sigma}\Bigl(\frac{1}{(1+|y-x_i|)^{\alpha+\beta-\sigma}}+
\frac{1}{(1+|y-x_j|)^{\alpha+\beta-\sigma}}\Bigr).
\]

\end{lemma}

\begin{proof}

Let $d_{ij}=|x_i-x_j|$.  If $y\in  B_{\frac12 d_{ij}}(x_i)$, then

\[
|y-x_j|\ge \frac12 |x_j-x_i|,\quad |y-x_j|\ge \frac12 |y-x_i|,
\]
which gives

\[
g_{ij}\le \frac{C}{|x_i-x_j|^\sigma}\frac{1}{(1+|y-x_i|)^{\alpha+\beta-\sigma}}, \quad
y\in  B_{\frac12 d_{ij}}(x_i).
\]

Similarly, we can prove

\[
g_{ij}\le \frac{C}{|x_i-x_j|^\sigma}\frac{1}{(1+|y-x_j|)^{\alpha+\beta-\sigma}}, \quad
y\in  B_{\frac12 d_{ij}}(x_j).
\]

Now we consider  $y\in \R^N\setminus \bigl( B_{\frac12
d_{ij}}(x_i)\cup B_{\frac12 d_{ij}}(x_j)\bigr)$. Then we have

\[
|y-x_i|\ge \frac12 |x_j-x_i|,\quad |y-x_j|\ge \frac12 |x_j-x_i|.
\]

If $|y-x_i|\ge 2|x_i-x_j|$, then

\[
|y-x_j|\ge |y-x_i|-|x_i-x_j|\ge \frac12 |y-x_i|.
\]
As a result,

\[
g_{ij}\le \frac{C}{(1+|y-x_i|)^{\alpha+\beta}}\le \frac{C_1}{|x_i-x_j|^\sigma}
\frac{1}{(1+|y-x_i|)^{\alpha+\beta-\sigma}},
\]
because $|y-x_i|\ge \frac12 |x_j-x_i|$.

If $|y-x_i|\le 2|x_i-x_j|$, then

\[
g_{ij} \le \frac{1}{(1+|y-x_i|)^\alpha}
\frac{C}{|x_i-x_j|^{\beta}}\le \frac{C}{|x_i-x_j|^\sigma}
\frac{1}{(1+|y-x_i|)^{\alpha+\beta-\sigma}},
\]
because $|y-x_j|\ge \frac12 |x_j-x_i|$.

\end{proof}

\begin{lemma}\label{laa2}
For any constant $0<\sigma<N-2$, there is a constant $C>0$, such
that

\[
\int_{\R^N} \frac1{|y-z|^{N-2}}\frac1{(1+|z|)^{2+\sigma}}\,dz\le \frac C{(1+|y|)^{\sigma}}.
\]

\end{lemma}

\begin{proof}
The result is well known. For the sake of completeness, we give the
proof.

We just need to obtain the estimate for $|y|\ge 2$. Let
$d=\frac12|y|$. Then, we have

\[
\begin{split}
&\int_{B_d(0)} \frac1{|y-z|^{N-2}}\frac1{(1+|z|)^{2+\sigma}}\,dz\le \frac{C}{d^{N-2}}
\int_{B_d(0)} \frac1{(1+|z|)^{2+\sigma}}\,dz\\
\le &\frac{C}{d^{N-2}}d^{N-2-\sigma}\le \frac{C}{d^\sigma},
\end{split}
\]
and

\[
\int_{B_d(y)} \frac1{|y-z|^{N-2}}\frac1{(1+|z|)^{2+\sigma}}\,dz\le
\frac C{d^{2+\sigma}}\int_{B_d(y)} \frac1{|z-y|^{N-2}}\,dz
\le \frac{C}{d^\sigma}.
\]

Suppose that $z\in \R^N\setminus \bigl(B_d(0)\cup B_d(y)\bigr)$.
Then

\[
|z-y|\ge \frac12 |y|,\quad  |z|\ge \frac12 |y|.
\]

If $|z|\ge 2|y|$, then $|z-y|\ge |z|-|y|\ge \frac12|z|$. As a
result,

\[
\frac1{|y-z|^{N-2}}\frac1{(1+|z|)^{2+\sigma}}\le \frac{C}{|z|^{N-2} (1+|z|)^{2+\sigma}}.
\]

If $|z|\le 2|y|$, then

\[
\frac1{|y-z|^{N-2}}\frac1{(1+|z|)^{2+\sigma}}
\le \frac{C}{|y|^{N-2} (1+|z|)^{2+\sigma}}\le \frac{C_1}{|z|^{N-2} (1+|z|)^{2+\sigma}}.
\]
Thus, we have proved that

\[
\frac1{|y-z|^{N-2}}\frac1{(1+|z|)^{2+\sigma}}
\le \frac{C}{|z|^{N-2} (1+|z|)^{2+\sigma}},\quad
z\in \R^N\setminus \bigl(B_d(0)\cup B_d(y)\bigr),
\]
which, give

\[
\int_{\R^N\setminus \bigl(B_d(0)\cup B_d(y)\bigr)}\frac1{|y-z|^{N-2}}\frac1{(1+|z|)^{2+\sigma}}\,dz
\le \frac{C}{d^\sigma}.
\]

\end{proof}

Let recall that

\[
W_r(y)=\bigl(N(N-2)\bigr)^{\frac{N-2}4}\sum_{j=1}^k \frac{\Lambda^{\frac{N-2}2}}{(1+\Lambda^2|y-x_j|^2)^{\frac{N-2}2}}.
\]

\begin{lemma}\label{laa3}

Suppose that $N\ge 5$ and $\tau\in (0, 2)$. Then there is a small
$\theta>0$, such that
\[
\begin{split}
&\int_{\R^N}\frac1{|y-z|^{N-2}} W_r^{\frac4{N-2}}(z)\sum_{j=1}^k\frac1{(1+|z-x_j|)^{\frac{N-2}2+\tau}}\,dz\\
\le & C\sum_{j=1}^k\frac1{(1+|y-x_j|)^{\frac{N-2}2+\tau+\theta}}+o(1)\sum_{j=1}^k\frac1{(1+|y-x_j|)^{\frac{N-2}2+\tau}},
\end{split}
\]
where $o(1)\to 0$ as $k\to +\infty$.

\end{lemma}

\begin{proof}
 Firstly,  we consider $N\ge 6$. Then $\frac4{N-2}\le 1$. Thus

\[
W_r^{\frac4{N-2}}(z)\le \sum_{i=1}^k\frac1{(1+|z-x_i|)^{4}}.
\]
So, we obtain

\[
\begin{split}
&\int_{\R^N}\frac1{|y-z|^{N-2}} W_r^{\frac4{N-2}}(z)\sum_{j=1}^k\frac1{(1+|z-x_j|)^{\frac{N-2}2+\tau}}\,dz\\
\le & \sum_{j=1}^k\int_{\R^N}\frac1{|y-z|^{N-2}} \frac1{(1+|z-x_j|)^{4+\frac{N-2}2+\tau}}\,dz\\
&+\sum_{j=1}^k\sum_{i\ne j}\int_{\R^N}\frac1{|y-z|^{N-2}}
\frac1{(1+|z-x_i|)^{4}} \frac1{(1+|z-x_j|)^{\frac{N-2}2+\tau}}\,dz.
\end{split}
\]

By Lemma~\ref{laa2}, if $\theta>0$ is so small that
$\frac{N-2}2+\tau+\theta<N-2$, then

\[
\begin{split}
&
\int_{\R^N}\frac1{|y-z|^{N-2}} \frac1{(1+|z-x_j|)^{4+\frac{N-2}2+\tau}}\,dz\\
\le &
\int_{\R^N}\frac1{|y-z|^{N-2}} \frac1{(1+|z-x_j|)^{2+\frac{N-2}2+\tau+\theta}}\,dz\le
\frac C{(1+|y-x_j|)^{\frac{N-2}2+\tau+\theta}}.
\end{split}
\]

On the other hand, it follows from Lemmas~\ref{laa1} and \ref{laa2}
that for $i\ne j$,

\[
\begin{split}
&\int_{\R^N}\frac1{|y-z|^{N-2}}
\frac1{(1+|z-x_i|)^{4}} \frac1{(1+|z-x_j|)^{\frac{N-2}2+\tau}}\,dz\\
\le & \frac{C}{|x_i-x_j|^2} \int_{\R^N}\frac1{|y-z|^{N-2}}\Bigl(
\frac1{(1+|z-x_i|)^{2+\frac{N-2}2+\tau}}+ \frac1{(1+|z-x_j|)^{2+\frac{N-2}2+\tau}}\Bigr)\,dz\\
\le & \frac{C}{|x_i-x_j|^2}
\Bigl(
\frac1{(1+|y-x_i|)^{\frac{N-2}2+\tau}}+ \frac1{(1+|y-x_j|)^{\frac{N-2}2+\tau}}\Bigr).
\end{split}
\]

Noting that

\[
\sum_{j\ne i}\frac1{|x_i-x_j|^2}\le \frac{C k^2}{\mu^2}\sum_{j=1}^k \frac1{j^2}=o(1),
\]
we obtain

\[
\begin{split}
&\sum_{j=1}^k\sum_{i\ne j}\int_{\R^N}\frac1{|y-z|^{N-2}}
\frac1{(1+|z-x_i|)^{4}} \frac1{(1+|z-x_j|)^{\frac{N-2}2+\tau}}\,dz\\
= & o(1)\sum_{j=1}^k
\frac 1{(1+|y-x_j|)^{\frac{N-2}2+\tau}}.
\end{split}
\]

Suppose now that  $N=5$. Recall that

\[
\Omega_{j}=\bigl\{ y:\;y=(y',y'')\in \R^2\times\R^{N-2},\;
  \bigl\langle \frac {y'}{|y'|}, \frac{x_j}{|x_j|}\bigr\rangle \ge \cos\frac {\pi}k\bigr\}.
\]

For $z\in\Omega_1$, we have $|z-x_j|\ge |z-x_1|$. Using
Lemma~\ref{laa1}, we obtain

\[
\begin{split}
&\sum_{j=2}^k\frac1{(1+|z-x_j|)^{3}}\le \frac1{(1+|z-x_1|)^{2}}\sum_{j=2}^k\frac1{1+|z-x_j|}
\\
\le & \frac C{(1+|z-x_1|)^{2}} \sum_{j=2}^k\frac 1{|x_j-x_1|}\le
 \frac C {(1+|z-x_1|)^{2}}.
 \end{split}
\]
Thus,

\[
W_r^{\frac43}(z)\le \frac C{(1+|z-x_1|)^{\frac 83}}.
\]
As a result, for $z\in\Omega_1$, using Lemma~\ref{laa1} again, we
find

\[
\begin{split}
&W_r^{\frac43}(z)\sum_{j=1}^k\frac1{(1+|z-x_j|)^{\frac 32+\tau}}\\
\le & \frac C{(1+|z-x_1|)^{\frac83+\frac 32+\tau}}+\frac C{(1+|z-x_1|)^{\frac13+2+\frac 32+\tau}}\sum_{j=2}^k
\frac1{|x_j-x_1|^{\frac13}}\\
\le & \frac C{(1+|z-x_1|)^{\frac83+\frac 32+\tau}}+\frac k{\mu^{\frac13}}
\frac C{(1+|z-x_1|)^{\frac13+2+\frac 32+\tau}}\\
\le & \frac C{(1+|z-x_1|)^{\frac13+2+\frac 32+\tau}},
\end{split}
\]
since

\[
\frac k{\mu^{\frac13}}\le \frac {Ck}{k^{\frac 3{3-m}\frac13}}\le C.
\]
So, we obtain

\[
\begin{split}
&\int_{\Omega_1}\frac1{|y-z|^{3}}W_r^{\frac43}(z)\sum_{j=1}^k\frac1{(1+|z-x_j|)^{\frac 32+\tau}}\,dz\\
\le & \int_{\Omega_1}\frac1{|y-z|^{3}} \frac C{(1+|z-x_1|)^{\frac13+2+\frac 32+\tau}}\,dz\le
\frac C{(1+|y-x_1|)^{\frac13+\frac 32+\tau}}.
\end{split}
\]
which gives

\[
\begin{split}
&\int_{\Omega}\frac1{|y-z|^{3}}W_r^{\frac43}(z)\sum_{j=1}^k\frac1{(1+|z-x_j|)^{\frac 32+\tau}}\,dz\\
=&
\sum_{i=1}^k \int_{\Omega_i}\frac1{|y-z|^{3}}W_r^{\frac43}(z)\sum_{j=1}^k\frac1{(1+|z-x_j|)^{\frac 32+\tau}}\,dz\\
\le &  \sum_{i=1}^k
\frac C{(1+|y-x_i|)^{\frac13+\frac 32+\tau}}.
\end{split}
\]

\end{proof}

\end{document}